\documentclass[11pt]{amsart}
\usepackage{a4}
\usepackage{amssymb}
\newtheorem{theorem}{Theorem}
\theoremstyle{plain}
\newtheorem{lemma}[theorem]{Lemma}

\theoremstyle{remark}

\newtheorem{remark}[theorem]{Remark}

\numberwithin{equation}{section}

\DeclareMathOperator{\sign}{sign}
\DeclareMathOperator{\supp}{supp}

\def\R{\mathbb R}

\def\underset#1\to#2{\mathop{#2}\limits_{#1}{ }}
\def\overset#1\to#2{\mathop{#2}\limits^{#1}{ }}
\begin{document}
\title{\bf A note on type of weak-$L^1$ and weak-$\ell^1$ spaces}


\author[Kami\'nska]{Anna Kami\'{n}ska}
\address[Kami\'nska]{Department of Mathematical Sciences,
The University of Memphis, TN 38152-3240, U.S.A.}
\email{kaminska@memphis.edu}

\date{}
\maketitle

\begin{abstract} We present a direct proof of the fact that the weak-$L^1$ and weak-$\ell^1$ spaces do not have type $1$. 
\end{abstract}

It has been known for some time that the weak -$L^1$ space is not normable, that is there does not exist a norm equivalent to the standard quasi-norm $\|f\|_{1,\infty}$ in the weak-$L^1$ space \cite{CRS}. In \cite[Proposition 2.3]{KalKam} it was proved more, namely that the weak-$L^1$ space does not have type $1$. This was obtained indirectly as a corollary of more general investigations. Here we present a direct proof by constructing suitable sequences of functions that contradicts type $1$ property in weak-$L^1$ or weak-$\ell^1$ spaces.

Let $r_n : [0,1]\rightarrow \Bbb R, n\in\Bbb N$, be Rademacher
functions, that is $r_n(t) = \sign \,(\sin 2^n\pi t)$. A quasi-Banach space
$(X, \|\cdot\|)$ has type $1$ \cite{Kal1, LT} if there is a constant $K>0$  such that,
for any choice of finitely many vectors $x_1,\dots ,x_n$ from $X$,
$$
\int_0^1 \Big\| \sum_{k=1}^n r_k(t)x_k \Big\|dt \le K \sum_{k=1}^n
\|x_k\|.
$$
Clearly if $X$ is a normable space then  $X$ has type $1$. For theory of quasi-Banach spaces see \cite{KPR}.

If $f$ is a~real-valued measurable function on $I$, where $I=(0,1)$ or $I=(0,\infty)$,   then
we define the \textit{distribution function} of $f$ by
$d_f(\lambda)=\left|\{x\in\R:\,|f(x)|>\lambda\}\right|$
for each $\lambda\ge 0$, where $|\cdot |$ denotes the Lebesgue measure on $\mathbb{R}$, and
the \textit{decreasing rearrangement} of $f$ is defined as
\begin{equation*}
f^*(t)=\inf\left\{s>0:d_f(s)\le t\right\},\quad
t\in I.
\end{equation*}
The {\it weak-$L^1$} space on $I$, called also the {\it Marcinkiewicz} space and denoted by $L_{1,\infty}(I)$ \cite{BS, KPS}, is the collection of all real valued  measurable functions on $I$ such that 
\[
\|f\|_{1,\infty}= \sup_{t\in I} t\,f^*(t)<\infty.
\]
The space  $L_{1,\infty}(I)$  equipped with the quasi-norm $\|\cdot\|_{1,\infty}$ is complete.

Analogously we define a sequence weak-$\ell^1$ space.
Given a bounded real-valued sequence $x=\{x(n)\}$,
consider the function $f(t) = \sum_{n=1}^\infty
x(n)\chi_{[n-1,n)}(t),\, t\ge 0,$ and define a decreasing rearrangement $x^* = \{x^*(n)\}$ of the sequence $x$ as follows
\[
x^*(n) = f^*(n-1),\ \ \ n\in \Bbb N.
\]
Then the weak-$\ell^1$ space denoted as $\ell_{1,\infty}$ consists of
all sequences $x=\{x(n)\}$ such that
\[
\|x\|_{1,\infty} = \sup_n n\,x^*(n)<\infty,
\]
and $\ell_{1,\infty}$ equipped with $\|\cdot\|_{1,\infty}$ is a quasi-Banach space.

\begin{lemma}\label{L:type1} For every $n\in \Bbb N$ there exists a sequence
$(g_{jn})_{j=1}^n\subset L_{1,\infty}(0,1)$ such that 
\[
 \|g_{j(n-1)}\|_{1,\infty} \le 1, \ \ \ \  n\in \Bbb N, \ \ j=1,\dots,n,
\]
and  for sufficiently large  $n\in \Bbb N$ and every choice of signs $\eta_j=\pm 1,\, j=1,\dots,n$,
\[
\frac12  n\log{n} \le \Big\|\sum_{j=1}^{n-1} \eta_j g_{j(n-1)}\Big\|_{1,\infty} \le  n \log{n}.
\]
\end{lemma}

\begin{proof} Let $k,n\in\Bbb N$ and $i=1,\dots, n-1$. Define for $t\in (0,1)$,
\[
f_{ki}(t) = \frac{1}{t+\frac{1}{n^{k-1}} - \frac{i}{n^k}}\chi_{\left(\frac{1}{n^k}, \frac{i}{n^k}\right]}(t)
+ \frac{1}{t-\frac{i-1}{n^k}} \chi_{\left(\frac{i}{n^k}, \frac{1}{n^{k-1}}\right]}(t).
\]
Setting
\[
F_k = \sum_{j=1}^{n-1} f_{kj},
\]
we have
\[
F_k(t) = \sum_{i=1}^{n-1}\Big(\sum_{j=1}^{n-1} \frac{1}{t+\frac{j-i}{n^k}}\Big)
\chi_{\left(\frac{i}{n^k}, \frac{i+1}{n^k}\right]}(t), \ \ \ t\in (0,1).
\]
We will show that for all $k\in\Bbb N$ and all $n\ge 10$,
\[
\frac12 n\log{n} \le \|F_k\|_{1,\infty} \le n\log{n}.
\]
In fact, if $\frac{i}{n^k} < t \le \frac{i+1}{n^k},\, i=1,\dots, n-1$, then
\[
n^k(\log{n} -1) \le F_k(t) = \sum_{j=1}^{n-1} \frac{1}{t+\frac{j-i}{n^k}} \le n^k \log{n}.
\]
Hence for all $0<t\le \frac{1}{n^{k-1}}\left(1-\frac{1}{n}\right)$,
\[
n^k(\log{n}-1) \le F_k^*(t) \le n^k \log{n},
\]
and so 
\[
\|F_k\|_{1,\infty} \le \frac{1}{n^{k-1}} \left(1-\frac{1}{n}\right) n^k \log{n} \le n\log{n}
\]
and for all $n\ge 10$,
\[
\|F_k\|_{1,\infty} \ge \frac{1}{n^{k-1}} \left(1-\frac{1}{n}\right) n^k(\log{n} -1) \ge \frac12 n\log{n}.
\] 

\par Let $m=1,\dots, 2^{n-1}$ and $\varepsilon^m = (\varepsilon_1^m,\dots,\varepsilon^m_{n-1})$ be
a sequence of signs $\varepsilon^m_j = \pm 1,\, j=1,\dots,n-1$. We assume here that $\varepsilon^{m_1}\ne
\varepsilon^{m_2}$ if $m_1\ne m_2$. Define now the following functions
\[
G_m = \sum_{j=1}^{n-1} \varepsilon_j^m f_{mj}.
\]
Notice that the supports of $G_m$ are disjoint and that $G_{m_0} = F_{m_0}$ whenever $\varepsilon_j^{m_0} =1$
for $j=1,\dots, n-1$. Therefore
\[
\Big\| \sum_{m=1}^{2^{n-1}} G_m\Big\|_{1,\infty} \ge \|G_{m_0}\|_{1,\infty} =
\|F_{m_0}\|_{1,\infty} \ge \frac12 n\log{n}.
\]
On the other hand observe that
\[
\sum_{m=1}^{2^{n-1}} F_m \le \sum_{m=1}^{2^{n-1}} n^m \log{n}\chi_{\left(\frac{1}{n^m}, \frac{1}{n^{m-1}}\right]}.
\]
Then
\[
\Big(\sum_{m=1}^{2^{n-1}} F_m\Big)^* \le \sum_{m=1}^{2^{n-1}} n^m\log{n}
\chi_{\left(\frac{1}{n^m} - \frac{1}{n^{2^{n-1}}}, \frac{1}{n^{m-1}} - \frac{1}{2^{2^{n-1}}}\right]},
\]
and so
\[
\Big\|\sum_{m=1}^{2^{n-1}} F_m \Big\|_{1,\infty} \le \max_{m=1,\dots,2^{n-1}}
\sup_{t\in\left(n^{-m} - n^{-2^{n-1}}, n^{-m+1} - 2^{-2^{n-1}}\right]}
tn^m\log{n} \le n\log{n}.
\]
Hence
\[
\Big\| \sum_{m=1}^{2^{n-1}} G_m \Big\|_{1,\infty} \le \Big\|\sum_{m=1}^{2^{n-1}} F_m\Big\|_{1,\infty}
\le n\log{n}.
\]
Now, let for $j=1,\dots,n-1$, $n\in \mathbb{N}$,
\[
g_{j(n-1)} = \sum_{m=1}^{2^{n-1}} \varepsilon^m_j f_{mj}.
\]
For any $\eta_j=\pm 1,\, j=1,\dots,n-1$, we have
\[
\sum_{j=1}^{n-1} \eta_j g_{j(n-1)} = \sum_{j=1}^{n-1}\eta_j\Big(\sum_{m=1}^{2^{n-1}} \varepsilon^m_j f_{mj}\Big)=
\sum_{m=1}^{2^{n-1}}\Big(\sum_{j=1}^{n-1} \eta_j \varepsilon_j^m f_{mj}\Big).
\]
Setting now $\alpha_j^m = \eta_j\varepsilon_j^m,\, m=1,\dots, 2^{n-1},\, j=1,\dots, n-1$, we get
\[
\sum_{j=1}^{n-1} \eta_j g_{j(n-1)} = \sum_{m=1}^{2^{n-1}}\Big(\sum_{j=1}^{n-1} \alpha_j^m f_{mj}\Big)
=
\sum_{m=1}^{2^{n-1}} G_m.
\]
Hence by the previous inequalities, for every choice of signs $\eta_j=\pm 1$ and for sufficiently large $n$,
we have
\begin{equation}\label{eq:11}
\frac12 n\log{n} \le \Big\|\sum_{j=1}^{n-1} \eta_j g_{j(n-1)} \Big\|_{1,\infty} \le n\log{n}.
\end{equation}
The functions $f_{mj}$ have disjoint supports with respect to $m = 1,\dots,2^{n-1}$ for each $j=1,\dots,n-1$. Hence
\[
|g_{j(n-1)}| = \sum_{m=1}^{2^{n-1}} f_{mj}\ \ \ \text{and}\ \ \ \ \supp{|g_{j(n-1)}|}= (n^{-2^{n-1}}, 1].
\]
It follows in view of the construction of the sequence $(f_{mj})$ that for $j=1,\dots,n-1$, $t\in(0,1)$ and $n\in \mathbb{N}$ we have
\[
g_{j(n-1)}^*(t) = \frac{1}{t+n^{-2^{n-1}}}\chi_{\left(0, 1- n^{-2^{n-1}}\right]}(t).
\]
Hence for all $j=1,\dots, n-1$ and $n\in \mathbb{N}$,
\begin{equation}\label{eq:12}
\|g_{j(n-1)}\|_{1,\infty}= \sup_{t\in (0,1)}\, t\, g_{j(n-1)}^*(t) = \frac{t}{t+n^{-2^{n-1}}}\chi_{\left(0, 1- n^{-2^{n-1}}\right]}(t)\le 1.
\end{equation}

In view of (\ref{eq:11}) and (\ref{eq:12}) the proof is completed.

\end{proof}

\begin{remark}\label{R:1} Lemma \ref{L:type1} remains also true for the sequence space $\ell_{1,\infty}$.
\end{remark}

\begin{theorem} \label{T:type1-weakL1}
The spaces $L_{1,\infty}(I)$ and $\ell_{1,\infty}$ do not have type $1$.  In particular,
$L_{1,\infty}(I)$ and $\ell_{1,\infty}$ are not normable.
\end{theorem}

\begin{proof}
Applying Lemma \ref{L:type1} we get
\[
\frac{\int_0^1\Big\|\sum_{j=1}^{n} r_j(t) g_{jn}\Big\|_{1,\infty} dt}
{\sum_{j=1}^{n}\|g_{jn}\|_{1,\infty}} = \frac{2^{-n} \sum_{\eta_j=\pm 1}
\Big\|\sum_{j=1}^{n} \eta_j g_{jn}\Big\|_{1,\infty}}{\sum_{j=1}^{n} \|g_{jn}\|_{1,\infty}}
\ge \frac{(n+1)\log{(n+1)}}{2n} \rightarrow \infty,
\]
as $n\rightarrow\infty$, which shows that the space $L_{1,\infty}(I)$ does not have type $1$. By Remark \ref{R:1} the proof also holds for sequence case.
\end{proof}

\bigskip


\begin{thebibliography}{9}

\bibitem{BS}
C. Bennett and R Sharpley,
\emph{Interpolation of Operators}, Academic Press, 1988.

\bibitem{CRS}
M. J. Carro, J. A. Raposo and J. Soria, \emph{Recent Developments in the Theory of Lorentz Spaces and
Weighted Inequalities}, Mem. Amer. Math. Soc. \textbf{187}~(2007).


\bibitem{Kal1}
N. J. Kalton,
\emph{Convexity, type and the three space problem},
Studia Math. \textbf{59}~1981, 247--287.


\bibitem{KalKam}
N. J. Kalton and A. Kami\'nska, \emph{Type and order convexity of Marcinkiewicz and Lorentz spaces and applications},
Glasgow Math. J. \textbf{47}~(2005) 123--137.

\bibitem{KPR}
N. J. Kalton, N. T. Peck and J. W. Roberts,
\emph{ An F-space Sampler},
Cambridge University Press, Lect. Notes Series, Vol. \textbf{89}~(1984).



\bibitem{KPS}
S. G. Krein, Ju. I. Petunin and E. M. Semenov,
\emph{Interpolation of Linear Operators}, AMS
Translations of Math. Monog. \textbf{54}~(1982).

\bibitem{LT}
J. Lindenstrauss and L. Tzafriri, \emph{Classical Banach
Spaces II}, Springer-Verlag, 1979.



\end{thebibliography}
\end{document}